\documentclass[12pt]{amsart}
\usepackage{graphicx}
\usepackage{amsfonts,amssymb,amscd,amsmath,enumerate,verbatim,calc}

\newtheorem{Theorem}{Theorem}[section]
\newtheorem{Lemma}[Theorem]{Lemma}

\newtheorem{Remark}[Theorem]{Remark}

\newtheorem{Example}[Theorem]{Example}

\newtheorem{Definition}[Theorem]{Definition}
\newtheorem{Open Problem}[Theorem]{Open Problem}

\begin{document}

\title{Soft Rough Graphs}
\author{R. Noor, I. Irshad, I. Javaid*}
\keywords{Soft rough graph, vertex and edge induced soft rough
graph, soft rough tree, Cartesian product of soft rough graphs,
lexicographic product
of soft rough graphs, corona product of soft rough graphs, join of soft rough graphs.\\
\indent $^*$ Corresponding author: imran.javaid@bzu.edu.pk}
\address{Centre for advanced studies in Pure and Applied Mathematics,
Bahauddin Zakariya University Multan, Pakistan. \newline
Email:ramshanoorsherazi@gmail.com,
iiqra.bzu@gmail.com,imran.javaid@bzu.edu.pk}
\date{}
\maketitle

\begin{abstract}Soft set theory and rough set theory are mathematical tools to deal with
uncertainties. In \cite{FF2011}, authors combined these concepts and
introduced soft rough sets. In this paper, we introduce the concepts
of soft rough graphs, vertex and edge induced soft rough graphs and
soft rough trees. We define some products with examples in soft
rough graphs.
\end{abstract}
\section{Introduction and Preliminaries} Many problems in engineering,
computer science, social science, economics, medical science etc. have various
uncertainties. To deal with these uncertainties, researchers have
proposed a number of theories such as probability theory, fuzzy set
theory, soft set theory, rough set theory, vague set theory etc. In
\cite{MD1999}, Molodtsov introduced the concept of soft sets. Soft
set theory has applications in many different fields including the
smoothness of functions, game theory, operational research, Perron
integration, probability theory and measurement theory. Research on
soft set theory is growing rapidly \cite{MAS2015, MRA2003, NM,
TRB2014}. In \cite{PZ1982}, Pawlak introduced the concept of rough
sets. This theory deals with the approximation of an arbitrary
subset of a universe by two definable or observable subsets called
lower and upper approximations. It has been successfully applied to
pattern recognition, machine learning, intelligent systems,
inductive reasoning, image processing, knowledge discovery, signal
analysis, decision analysis, expert systems and many other fields
\cite{pawlk1, pawlk2, pawlk3, pawlk4}. In \cite{FF2011}, authors
introduced the concept of soft rough sets. Also, soft sets are
combined with fuzzy sets and rough sets in \cite{FF2010}. The
concept of rough soft group is defined in \cite{GJ2013}.

The concepts of soft graph and rough graph are introduced in
\cite{MAS2015} and \cite{HTK} respectively. In this paper, we
introduce the concept of soft rough graph, vertex and edge induced
soft rough graph and soft rough trees. We study AND and OR
operations in soft rough graphs. Also, we study the Cartesian,
lexicographic and corona product and join of two soft rough graphs.

Now we recall some notions related to soft sets and rough sets. Let
$U$ be the universe of discourse, $\mathcal{R}$ be an equivalence
relation on $U$ and $\mathcal{P}$ be the universe of all possible
parameters related to the objects in $U$. The power set of $U$ is
denoted by $P(U)$.
\begin{Definition}\cite{MD1999}A pair
$\mathcal{F}=(\mathfrak{F},\,\mathcal{A})$ is called soft set over
$U$, where $\mathcal{A} \subseteq$ $\mathcal{P}$ and  $\mathfrak{F}$
is a set valued mapping, $\mathfrak{F}:\mathcal{A} \rightarrow
P(U)$. The ordered pair $(U, \mathcal{P})$ is known as soft
universe.
\end{Definition}
\begin{Definition}\cite{PZ1982}The pair $(U, \mathcal{R})$ is
called Pawlak approximation space. For $\mathcal{X} \subseteq U$,
 $$\mathcal{R}_{*}(\mathcal{X})=\{x \in U: [x]_{\mathcal{R}} \subseteq U\}$$
 $$\mathcal{R}^{*}(\mathcal{X})=\{x \in U: [x]_{\mathcal{R}} \cap \mathcal{X} \neq \phi \} $$
 are called lower and upper approximations of the set $\mathcal{X}.$
 If $\mathcal{R}_{*}(\mathcal{X})=\mathcal{R}^{*}(\mathcal{X})$ then $\mathcal{X}$ is said to
 be definable otherwise $\mathcal{X}$ is called rough set.
\end{Definition}
\begin{Definition}\cite{FF2011}
Let $\mathcal{S}=(U, \mathcal{F} )$ be a soft approximation space.
Based on $\mathcal{S}$, the lower and upper soft rough
approximations of $\mathcal{X} \subseteq $ $U$ are defined
by:\[\mathfrak{F_{*}}(\mathcal{X})=\{u \in U: \exists \, a \, \in \,
\mathcal{A}\, | \,u \, \in \, \mathfrak{F}(a) \, \subseteq \,
\mathcal{X}\}\] and
\[\mathfrak{F^{*}}(\mathcal{X})=\{u \in  U: \exists \, a \, \in \,
\mathcal{A}\, | \,u \, \in \, \mathfrak{F}(a)\,\, and\,\,
\mathfrak{F}(a) \, \cap \,\mathcal{X} \neq \phi\}\] respectively. If
$\mathfrak{{F}_{*}}(\mathcal{X})=\mathfrak{{F}^{*}}(\mathcal{X})$
then $\mathcal{X}$ is said to be soft definable otherwise
$\mathcal{X}$ is called soft rough set, written as
$(\mathfrak{F}_*(\mathcal{X}),\mathfrak{F}^*(\mathcal{X}),\mathcal{A})$.
\end{Definition}
A graph $G=(V(G), E(G))$ is a pair, where $V(G)$ is the set of
vertices and $E(G)$ is the set of edges of $G.$ When there is no
ambiguity, we write $G=(V,E)$. If two vertices $x$ and $y$ are
adjacent in $G$, we write $xy\in E(G)$. Two edges are said to be
adjacent if they share a common vertex. A graph $H$ is a subgraph of
$G$ if $V(H)\subseteq V(G)$ and $E(H)\subseteq E(G)$. The open
neighborhood of $v$ in $G$, $N(v)$, is defined as $N(v)=\{u| uv \in
E(G)\}$ and $N[v]=\{v\}\cup N(v)$. The degree of $v$ is defined as
$|N(v)|$. A graph is called simple if it contains no loops or
multiple edges. The distance between two vertices $u,v$ of $G$,
$d(u,v)$, is the length of a shortest path between them. The
diameter of a graph $G$, $diam(G)$, is defined as
$diam(G)=\,$max$\{d(u,v)|u,v\in V(G)\}$.
\begin{Definition}\cite{MAS2015}
A graph
$\tilde{G}$=$(G,\,\mathfrak{F},\,\mathfrak{K},\,\mathcal{A})$ is
called soft graph if it
satisfies  the following conditions:\\
$1)$ $G=(V,\,E)$ is a simple graph.\\
$2)$ $\mathcal{A}$ is a non-empty set of parameters.\\
$3)$ $(\mathfrak{F},\,\mathcal{A})$ is a soft  set over $V$. \\
$4)$ $(\mathfrak{K},\,\mathcal{A})$ is a soft set over $E$. \\
$5)$  $\mathfrak{H}(x)=( \mathfrak{F}(x),\,\mathfrak{K}(x))$ is a
subgraph of G for all $x$ $\in$ $ \mathcal{A}.$
\par A soft graph can be represented by
  $ \tilde{G}=\langle\mathfrak{F},\,\mathfrak{K},\,\mathcal{A}\rangle=\{\mathfrak{H}(x)|\, x \in \mathcal{A}\}.$
  The set of all soft graphs of
$G$ is denoted by $\mathbf{S}\mathcal{G}(G).$
\end{Definition}

\section{Soft Rough Graphs}
Let $\mathcal{S}=(U,\,\mathcal{F})$ be a soft approximation space,
$\mathcal{A} \subseteq \mathcal{P}$ be any non-empty subset of
parameters and $\mathcal{X} \subseteq U$ be any non-empty subset of
$U.$ Let $(
\mathfrak{F_{*}}(\mathcal{X}),\,\mathfrak{F^{*}}(\mathcal{X}),\mathcal{A})$
be a soft rough set over $V$ with its lower and upper approximations
defined as:
\[\mathfrak{F_{*}}(\mathcal{X})=\{u \in U: \exists \, a \, \in \,
\mathcal{A}\, | \,u \, \in \, \mathfrak{F}(a) \, \subseteq \,
\mathcal{X}\}\] and
\[\mathfrak{F^{*}}(\mathcal{X})=\{u \in U: \exists \, a \, \in \,
\mathcal{A}\, | \,u \, \in \, \mathfrak{F}(a)\,\, and\,\,
\mathfrak{F}(a) \, \cap \, \mathcal{X} \neq \phi\}\] respectively,
where $(\mathfrak{F},\mathcal{A})$ be a soft set over $V$ with
$\mathfrak{F}:\mathcal{A} \rightarrow P(V)$ defined as
$\mathfrak{F}(x)=\{y\in V: xRy\}$. Let $(
\mathcal{K_{*}}(\mathcal{X}),\,\mathcal{K^{*}}(\mathcal{X}),\mathcal{A})$
be a soft rough set over $E$ with its lower and upper approximations
defined as:
\[\mathcal{K_{*}}(\mathcal{X})=\{e \in E: \exists \, a \, \in \,
\mathcal{A}\, | \,e \, \in \, \mathcal{K}(a) \, \subseteq \,
\mathcal{X}\}\] and
\[\mathcal{K^{*}}(\mathcal{X})=\{e \in E: \exists \, a \, \in \,
\mathcal{A}\, | \,e \, \in \, \mathcal{K}(a)\,\, and\,\,
\mathcal{K}(a) \, \cap \, \mathcal{X} \neq \phi\}\] respectively,
where $(\mathcal{K},\mathcal{A})$ be a soft set over $E$ with
$\mathcal{K}:\mathcal{A} \rightarrow P(E)$ defined as:
\[\mathcal{K}(x)=\{e \in E \,|\, e \, \subseteq \,\mathfrak{F}(x)\}.\]
\begin{Definition}
A graph
$\tilde{G}=(G,\mathfrak{F_{*}},\mathcal{K_{*}},\mathfrak{F^{*}},\mathcal{K^{*}},\mathcal{A},
\mathcal{X})$ is called soft rough graph if it satisfies the
following conditions:
\\ $1)$ $ G=(V,E)$ is a simple graph.\\ $2)$ $\mathcal{A}$ be a non-empty set of
parameters.\\ $3)$ $\mathcal{X}$ be any non-empty subset of $U.$
\\ $4)$ $(\mathfrak{F^{*}}(\mathcal{X}),\,\mathfrak{F_{*}}(\mathcal{X})
,\,\mathcal {A}) $ be a soft rough set over $V$.
\\
$5)$ $(\mathcal{K^{*}}(\mathcal{X}),\,\mathcal{K_{*}}(\mathcal{X}),\,\mathcal{A}) $  be a soft rough set over $E$.\\
$6)$
$\mathcal{H^{*}}(\mathcal{X})=(\mathfrak{F^{*}}(\mathcal{X}),\,\mathcal{K^{*}}(\mathcal{X}))$
and
$\mathcal{H_{*}}(\mathcal{X})=(\mathfrak{F_{*}}(\mathcal{X}),\,\mathcal{K_{*}}(\mathcal{X}))$
are subgraphs of $G.$
\par A soft rough graph can be represented by
\[\tilde{G}=\langle\mathfrak{F_{*}},\mathcal{K_{*}},\mathfrak{F^{*}},\mathcal{K^{*}},\mathcal{A},
\mathcal{X}\rangle=\{\mathcal{H_{*}}(\mathcal{X}),\,\mathcal{H^{*}}(\mathcal{X})\}.\]
The set of all soft rough graphs of $G$ is denoted by
$\mathcal{S}\mathfrak{R}\mathcal{G}(G).$
\end{Definition}
\begin{Definition} A soft rough graph
$\tilde{G}=\langle\mathfrak{F_{*}},\mathcal{K_{*}},
\mathfrak{F^{*}},\mathcal{K^{*}},\mathcal{A}, \mathcal{X}\rangle$
 is said to be lower vertex induced graph if
$$
\mathcal{H_{*}}(\mathcal{X})=(\mathfrak{F_{*}}(\mathcal{X}),\,\mathcal{K_{*}}(\mathcal{X}))=\langle\mathfrak{F_{*}}(\mathcal{X})\rangle
\,\,\, \mbox{for} \,\,\, \mathcal{X}\,\subseteq\, V ,$$ upper vertex
induced graph if $$
\mathcal{H^{*}}(\mathcal{X})=(\mathfrak{F^{*}}(X),\mathcal{K^{*}}(\mathcal{X}))=\langle\mathfrak{F^{*}}(\mathcal{X})\rangle
\,\,\,\mbox{for} \,\,\, \mathcal{X}\,\subseteq \, V ,$$lower edge
induced graph if
$$\mathcal{H_{*}}(\mathcal{X})=(\mathfrak{F_{*}}(\mathcal{X}),
\mathcal{K_{*}}(\mathcal{X}))=\langle\mathcal{K_{*}}(\mathcal{X})\rangle
\,\,\,
 \mbox{for} \,\,\, \mathcal{X}\,\subseteq\, V,$$
upper edge induced graph if
$$\mathcal{H^{*}}(\mathcal{X})=(\mathfrak{F^{*}}(\mathcal{X}),\mathcal{
K^{*}}(\mathcal{X}))=\langle\mathcal{K^{*}}(\mathcal{X})\rangle
 \,\,\, \mbox{for} \,\,\, \mathcal{X}\,\subseteq \, V .$$

\end{Definition}
\begin{Example}
Let $G=(V,E)$ be a graph with $V=\{v_1, v_2, v_3, v_4, v_5\} $ and
$E=\{e_1, e_2, e_3, e_4, e_5, e_6, e_7,e_8, e_9\} $ as shown in
\textup{FIGURE \ref{fig12}}.
\begin{figure}[!ht]
    \centerline
      {\includegraphics[width=4cm]{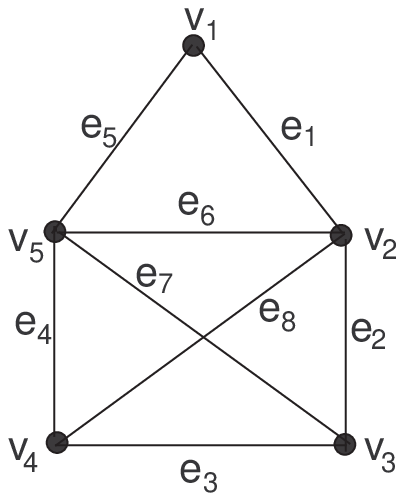}}
      \caption{}\label{fig12}
\end{figure}
 Let $\mathcal{A}=\{v_1,\,v_3\}$ be the set of parameters and
 $\mathfrak{F}:\mathcal{A}\rightarrow P(V)$ is defined as
\[ \mathfrak{F}(x)=\{y \in V: xRy \Leftrightarrow y \in N(x) \}.\]
Then $\mathfrak{F}(v_1)= \{v_2,\,v_5\},$
$\mathfrak{F}(v_3)=\{v_2,\,v_4,\,v_5\}.$ Let
$\mathcal{X}=\{v_1,\,v_2,\,v_5\}$ be a subset of $V.$ The lower and
upper soft rough approximations over $V$ are as follows:
$$\mathfrak{F_{*}}(\mathcal{X})=\{v_2,\,v_5\},\,\,
\mathfrak{F^{*}}(\mathcal{X})=\{v_2,\,v_4,\,v_5\}$$ respectively.
The lower and upper soft rough approximations over $E$ are as
follows:
$$\mathcal{K_{*}}(\mathcal{X})=\{e_6\,\},\,\,
\mathcal{K^{*}}(\mathcal{X})=\{e_4,\,e_6,\,e_8\}$$ respectively.
\begin{figure}[!ht]
    \centerline
      {\includegraphics[width=8cm]{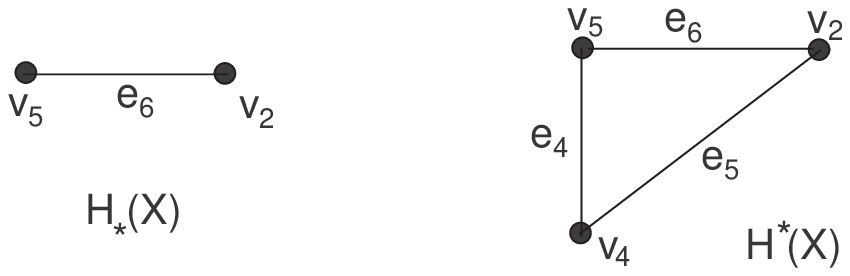}}
      \caption{}\label{fig13}
\end{figure}
Now
$\mathcal{H_{*}}(\mathcal{X})=(\mathfrak{F_{*}}(\mathcal{X}),\,\mathcal{K_{*}}(\mathcal
{X}))=(\{v_2,\,v_5\},\,\{e_6\})$ and
$\mathcal{H^{*}}(\mathcal{X})=(\mathfrak{F^{*}}(\mathcal{X}),\,\mathcal{K^{*}}(\mathcal{X}))=(\{v_2,\,v_4,\,v_5\},\,\{e_4,\,e_6,\,e_8\})$
as shown in \textup{FIGURE \ref{fig13}}.
\end{Example}
\begin{Definition}
Let
$\tilde{G_{1}}=\langle\mathfrak{F_{1*}},\mathcal{K}_{1*},\mathfrak{F}_{1}^{*},\mathcal{K}_{1}^{*},\mathcal{A},
\mathcal{X}\rangle$ and
$\tilde{G_{2}}=\langle\mathfrak{F_{2*}},\mathcal{K}_{2*},\mathfrak{F}_{2}^{*},\mathcal{K}_{2}^{*},\mathcal{B},
\mathcal{X}\rangle$ be two soft rough graphs of $G.$ Then
$\tilde{G_{2}}$ is
called soft rough subgraph of $\tilde{G_{1}}$ if the following conditions hold:\\
$1)$ $\mathcal{B} \subseteq \mathcal{A}.$ \\
$2)$ $\mathcal{H}_{2*}(\mathcal{X})=(\mathfrak{F}_{2*}(\mathcal{X}),
\mathcal{K}_{2*}(\mathcal{X}))$ is a subgraph
of\, $\mathcal{H}_{1*}(\mathcal{X})=(\mathfrak{F}_{1*}(\mathcal{X}),  \mathcal{K}_{1*}(\mathcal{X}) ).$\\
$3)$ $\mathcal{H}_{2}^{*}(
\mathcal{X})=(\mathfrak{F}_{2}^{*}(\mathcal{X}),
   \mathcal{K}_{2}^{*}(\mathcal{X}))$ is a subgraph of\, $\mathcal{H}
_{1}^{*}(\mathcal{X})=(\mathfrak{F}_{1}^{*}(\mathcal{X}),
\mathcal{K}_{1}^{*}(\mathcal{X})).$
\end{Definition}
\begin{Example}
Let $G=(V,E)$ be a graph with $V=\{v_1, v_2, v_3, v_4, v_5,
v_6.v_7\} $ and $E=\{e_1, e_2, e_3, e_4, e_5, e_6, e_7,e_8,
e_9,e_{10},e_{11}\} $ as shown in \textup{FIGURE \ref{fig14}}.
\begin{figure}[!ht]
    \centerline
      {\includegraphics[width=7cm]{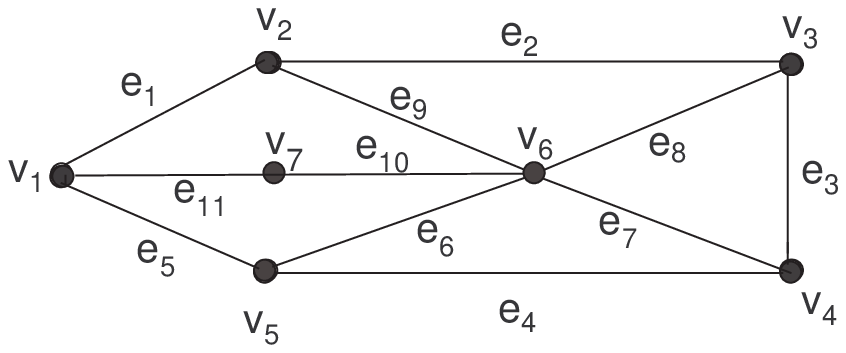}}
      \caption{}\label{fig14}
\end{figure}
 Let $\mathcal{B}=\{v_2,\,v_4\}$ be the set of parameters and
 $\mathfrak{F}_{2}:\mathcal{B}\rightarrow P(V)$ is defined as
$$\mathfrak{F}_{2}(x)=\{y \in V: xRy \Leftrightarrow y\in N(x)
\}.$$ Then $\mathfrak{F}_{2}(v_2)= \{v_1,\,v_3,\,v_6\},$
$\mathfrak{F}_{2}(v_4)=\{v_3,\,v_5,\,v_6\}.$ Let
$\mathcal{X}=\{v_1,\,v_3,\,v_{6}\}$ be a subset of $V.$ The lower
and upper soft rough approximations over $V$ are as follows:
$$\mathfrak{F}_{2*}(\mathcal{X})=\{v_1,\,v_3,\,v_6\},\,\,
\mathfrak{F}_{2}^{*}(\mathcal{X})=\{v_1,\,v_3,\,v_5,\,v_6\}$$
respectively. The lower and upper soft rough approximation over $E$
are as follows:
$$\mathcal{K}_{2*}(\mathcal{X})=\{e_8\,\},\,\,\mathcal{K}_{2}^{*}(\mathcal{X})=\{e_6,\,e_8\}$$
respectively. As $\mathcal{H}_{2*}(\mathcal{X}) ,\,
\mathcal{H}_{2}^{*}(\mathcal{X})$ are shown in \textup{FIGURE
\ref{fig15}}, hence $\tilde{G_{2}}=\{ \mathcal{H}_{2*}(\mathcal{X})
,\, \mathcal{H}_{2}^{*}(\mathcal{X})\}$ be the soft rough graph of
$G$.
\begin{figure}[!ht]
    \centerline
      {\includegraphics[width=8cm]{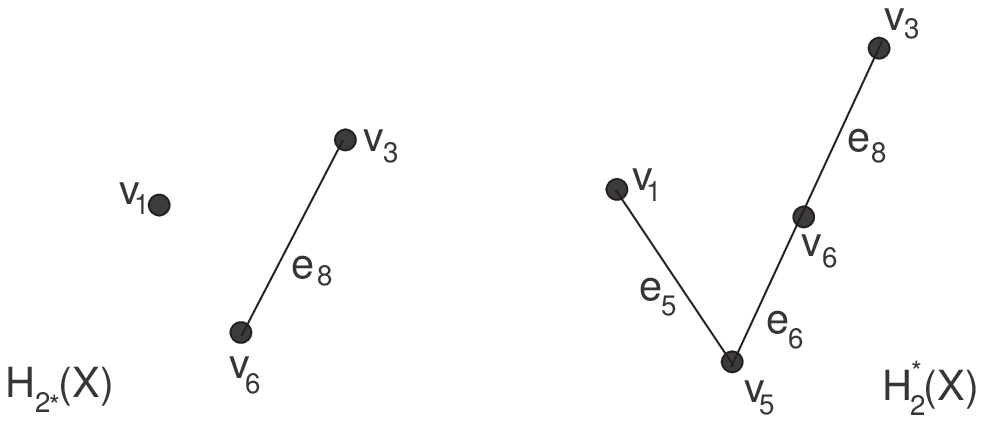}}
      \caption{}\label{fig15}
\end{figure}
Let $\mathcal{A}=\{v_2,\,v_3,\,v_4\}$ be the set of parameters such
that $\mathcal{B} \subseteq \mathcal{A}.$ Let
$\mathfrak{F}_{1}:\mathcal{A}\rightarrow P(V)$ defined as
\[ \mathfrak{F}_{1}(x)=\{y \in V: xRy \Leftrightarrow y\in N(x)
\}.\] Then $\mathfrak{F}_{1}(v_2)= \{v_1,\,v_3,\,v_6\},$
$\mathfrak{F}_{1}(v_4)=\{v_3,\,v_5,\,v_6\},$
$\mathfrak{F}_{1}(v_3)=\{v_2,\,v_4,\,v_6\}.$ The lower and upper
soft rough approximations over $V$ are as follows:
\[\mathfrak{F}_{1*}(\mathcal{X})=\{v_1,\,v_3,\,v_6\},\,\,
\mathfrak{F}_{1}^{*}(\mathcal{X})=\{v_1,\,v_2,\,v_3,\,v_4,\,v_5,\,v_6\}\]
respectively. The lower and upper soft rough approximations over $E$
are as follows:
\[\mathcal{K}_{1*}(\mathcal{X})=\{e_8\,\},\,\,
\mathcal{K}_{1}^{*}(\mathcal{X})=\{e_1,\,e_2,\,e_3,\,e_4,e_5,e_6,e_7,e_8,e_9\}\]
respectively. As $\mathcal{H}_{1*}(\mathcal{X}) ,\,
\mathcal{H}_{1}^{*}(\mathcal{X})$ are shown in \textup{FIGURE
\ref{fig16}}, hence $\tilde{G_{1}}=\{ \mathcal{H}_{1*}(\mathcal{X})
,\, \mathcal{H}_{1}^{*}(\mathcal{X}) \}$ be the soft rough graph of
$G$.
\begin{figure}[!ht]
    \centerline
      {\includegraphics[width=8cm]{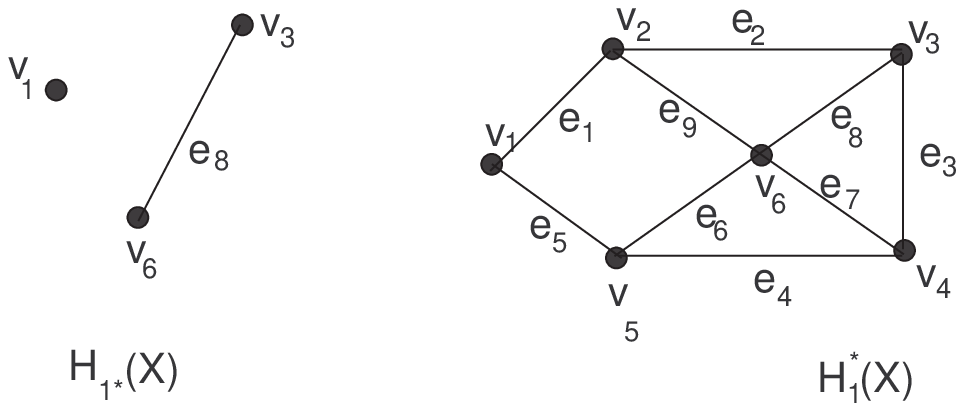}}
      \caption{}\label{fig16}
\end{figure}
Since $\mathcal{B} \subseteq \mathcal{A},$
$\mathcal{H}_{2*}(\mathcal{X})$ is a subgraph of
$\mathcal{H}_{1*}(\mathcal{X})$ and
 $\mathcal{H}_{2}^{*}(\mathcal{X})$ is
a subgraph of $\mathcal{H}_{1}^{*}(\mathcal{X})$. So,
$\tilde{G_{2}}$ is a soft rough subgraph of  $\tilde{G_{1}}$.
\end{Example}
\begin{Theorem}
Let
$\tilde{G_1}=\langle\mathfrak{F_{1*}},\,\mathfrak{F_{1}^{*}},\,\mathcal{K}
_{1*},\ \mathcal{K}_{1}^{*},\,\mathcal{A}, \mathcal{X}\rangle$ and
$\tilde{G_2}=\langle
\mathfrak{F_{2*}},\,\mathfrak{F_{2}^{*}},\,\mathcal{K}_{2*},\,
\mathcal{K}_{2}^{*},\,\mathcal{B}, \mathcal{X}\rangle$  be two soft
rough graphs of $G.$ Then $\tilde{G_2}$ is a soft rough subgraph of
$\tilde{G_1}$ if and only if $ \mathfrak{F}_{2*}(\mathcal{X})
\subseteq \mathfrak{F}_{1*}(\mathcal{X}),$
\,$\mathfrak{F}_{2}^{*}(\mathcal{X})\subseteq \mathfrak
{F}_{1}^{*}(\mathcal{X}),$ \,$\mathcal{K}_{2*}(\mathcal{X})\subseteq
\mathcal{K}_{1*}(\mathcal{X})$ and
$\mathcal{K}_{2}^{*}(\mathcal{X})\subseteq
\mathcal{K}_{1}^{*}(\mathcal{X})$.
\end{Theorem}
\begin{proof}
Given $\tilde{G_1}$ and $\tilde{G_2}$ be two soft rough graphs of
$G.$ Suppose $\tilde{G_2}$ is a soft rough subgraph of
$\tilde{G_1}.$ Then
by the definition of soft rough subgraph:\\ $1)$ $ \mathcal{B} \subseteq \mathcal{A}.$ \\
$2)$
$\mathcal{H}_{2*}(\mathcal{X})=(\mathfrak{F}_{2*}(\mathcal{X}),\,\mathcal{K}_{2*}(\mathcal{X}))$
is a subgraph of $
\mathcal{H}_{1*}(\mathcal{X})=(\mathfrak{F}_{1*}(\mathcal{X}),\,\mathcal{K}_{1*}(\mathcal{X}))$.
\\ $3)$ $ \mathcal{H}_{2}^{*}(\mathcal{X})=(\mathfrak{F}_{2}^{*}(\mathcal{X}),\,\, \mathcal{K}_{2}^{*}(\mathcal{X}))$ is a subgraph of
 $\mathcal{H}_{1}^{*}(\mathcal{X})=(\mathfrak{F}_{1}^{*}(\mathcal{X}),\,\,\mathcal{K}_{1}^{*}(\mathcal{X}))$.

Since $\mathcal{H}_{2*}(\mathcal{X})$ is a subgraph of
$\mathcal{H}_{1*}(\mathcal{X})$ and $
\mathcal{H}_{2}^{*}(\mathcal{X})$ is a subgraph of
$\mathcal{H}_{1}^{*}(\mathcal{X})$. Thus
\,$\mathfrak{F}_{2*}(\mathcal{X})\subseteq
\mathfrak{F}_{1*}(\mathcal{X})$,\,
$\mathcal{K}_{2*}(\mathcal{X})\subseteq
\mathcal{K}_{1*}(\mathcal{X})$,\,
$\mathfrak{F}_{2}^{*}(\mathcal{X})\subseteq
\mathfrak{F}_{1}^{*}(\mathcal{X})$ and
$\mathcal{K}_{2}^{*}(\mathcal{X})\subseteq
\mathcal{K}_{1}^{*}(\mathcal{X})$.

Conversely suppose that $\mathfrak{F}_{2*}(\mathcal{X})\subseteq
\mathfrak{F}_{1*}(\mathcal{X})$,
$\mathcal{K}_{2*}(\mathcal{X})\subseteq
\mathcal{K}_{1*}(\mathcal{X})$, $\mathfrak{F}_{2}^{*}(\mathcal{X})
\subseteq \mathfrak{F}_{1}^{*}(\mathcal{X})$ and
$\mathcal{K}_{2}^{*}(\mathcal{X}) \subseteq
\mathcal{K}_{1}^{*}(\mathcal{X})$. Since $\tilde{G_1}$ is a soft
rough graph of $G,$ $\mathcal{H}_{1*}(\mathcal{X})$ and
$\mathcal{H}_{1}^{*}(\mathcal{X})$ be the subgraphs of $G$. Since
$\tilde{G_2}$ is a soft rough graph of $G,$
$\mathcal{H}_{2*}(\mathcal{X})$ and
$\mathcal{H}_{2}^{*}(\mathcal{X})$ be subgraphs of $G$. Thus
$\mathcal{H}_{2*}(\mathcal{X})=(\mathfrak{F}_{2*}(\mathcal{X}),\,\mathcal{K}_{2*}(\mathcal{X}))$
is a subgraph of $
\mathcal{H}_{1*}(\mathcal{X})=(\mathfrak{F}_{1*}(\mathcal{X}),\,\mathcal{K}_{1*}(\mathcal{X}))$
and $
\mathcal{H}_{2}^{*}(\mathcal{X})=(\mathfrak{F}_{2}^{*}(\mathcal{X}),\,\mathcal{K}_{2}^{*}(\mathcal{X}))$
is a subgraph of
$\mathcal{H}_{1}^{*}(\mathcal{X})=(\mathfrak{F}_{1}^{*}(\mathcal{X}),\,\mathcal{K}_{1}^{*}(\mathcal{X}))$.
Hence $\tilde{G_2}$ is a soft rough subgraph of
$\tilde{G_1}.$\end{proof}

\begin{Definition}Let
$\tilde{G}=\langle\mathfrak{F_{*}},\mathcal{K}_{*},\mathfrak{F}^{*},\mathcal{K}^{*},\mathcal{A},
\mathcal{X}\rangle$ be a soft rough graph of $G$. Then $\tilde{G}$
is called a soft rough tree if $\mathcal{H}_*(\mathcal{X})$ and
$\mathcal{H}^*(\mathcal{X})$ are trees.
\end{Definition}
\begin{figure}[!ht]
    \centerline
      {\includegraphics[width=5cm]{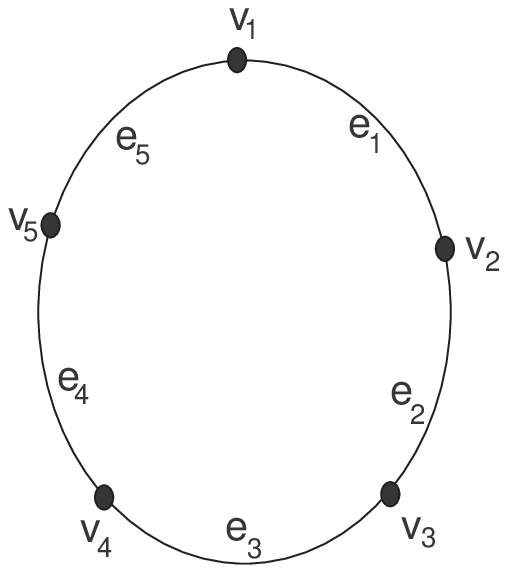}}
      \caption{}\label{treefig}
\end{figure}
\begin{Example}\label{tree}Let $G=(V,E)$ be a graph
with $V=\{v_1, v_2, v_3, v_4, v_5\}$ and $E=\{e_1,
e_2,e_3,e_4,e_5\}$ as shown in \textup{FIGURE \ref{treefig}}. Let
$\mathcal{A}=\{v_1, v_2\}$ be the set of parameters and
$\mathfrak{F}:\mathcal{A}\rightarrow P(V)$ is defined as
$$\mathfrak{F}(x)=\{y: xRy \Leftrightarrow d(x,y)= diam(G)\}.$$
Then $\mathfrak{F}(v_1)=\{v_3,v_4\}$,
$\mathfrak{F}(v_2)=\{v_4,v_5\}$. Let $\mathcal{X}=\{v_2,v_3,v_4\}$
then the lower and upper soft rough approximations over $V$ are as
follows:
$$\mathfrak{F}_*(X)=\{v_3, v_4\},\,\,\mathfrak{F}^*(X)=\{v_3, v_4,
v_5\},$$ respectively. The lower and upper soft rough approximations
over $E$ are as follows:
$$\mathcal{K}_*(X)=\{e_3\},\,\,\mathcal{K}^*(X)=\{e_3,e_4\},$$
respectively. As shown in \textup{FIGURE \ref{tree2}},
$\mathcal{H}_*(\mathcal{X})$ and $\mathcal{H}^*(\mathcal{X})$ are
trees. Hence $\tilde{G}$ is a soft rough tree.
\begin{figure}[!ht]
    \centerline
      {\includegraphics[width=7cm]{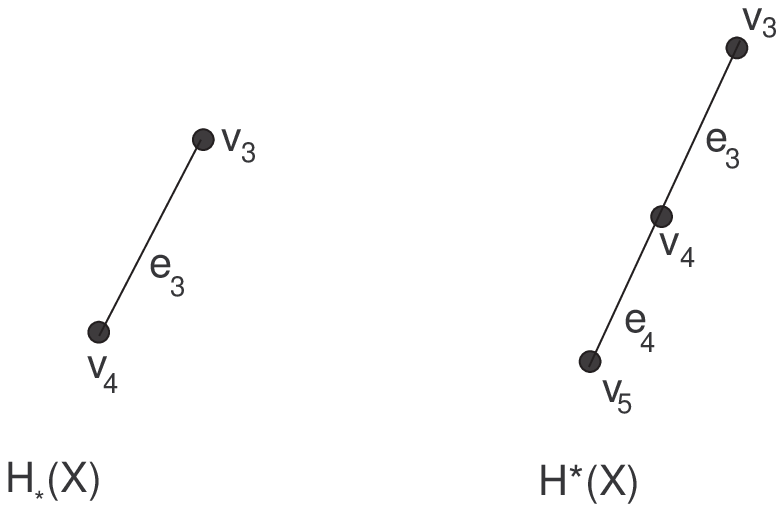}}
      \caption{}\label{tree2}
\end{figure}
\end{Example}
\begin{Lemma}Let $T$ be a tree. Then any soft rough graph $\tilde{G}$
of $T$ is a soft rough tree.
\end{Lemma}
\begin{proof} Since $T$ is a tree so for any soft rough graph
$\tilde{G}=\langle\mathfrak{F_{*}},\,\mathfrak{F^{*}},\,\mathcal{K}
_{*},\ \mathcal{K}^{*},\,\mathcal{A}, \mathcal{X}\rangle$ of $T$,
$\mathcal{H}_*(\mathcal{X})$ and $\mathcal{H}^*(\mathcal{X})$ are
trees.
\end{proof}
Note that converse of above Lemma is not true in general. As in
Example \ref{tree}, $\mathcal{H}_*(\mathcal{X})$ and
$\mathcal{H}^*(\mathcal{X})$ are trees but $G$
 is a cycle of order $5.$
 \begin{Remark}Let
$\tilde{G_1}=\langle\mathfrak{F_{1*}},\,\mathfrak{F_{1}^{*}},\,\mathcal{K}
_{1*},\ \mathcal{K}_{1}^{*},\,\mathcal{A}, \mathcal{X}\rangle$ be a
soft rough tree of $G$. If $\tilde{G_2}=\langle
\mathfrak{F_{2*}},\,\mathfrak{F_{2}^{*}},\,\mathcal{K}_{2*},\
\mathcal{K}_{2}^{*},\,\mathcal{B}, \mathcal{Y}\rangle$ is a soft
rough subgraph of $\tilde{G_1}$ then $\tilde{G_2}$ is a soft rough
tree of $G$.
\end{Remark}
\begin{Definition}
Let
$\tilde{G_1}=\langle\mathfrak{F}_{1*},\,\mathfrak{F}_{1}^{*},\,\mathcal{K}_{1*},\,
\mathcal{K}_{1}^{*},\,\mathcal{A},\, \mathcal{X}\rangle$ and
$\tilde{G_2}=\langle\mathfrak{F}_{2*},\,\mathfrak{F}_{2}^{*},\,\mathcal{K}_{2*},\,
\mathcal{K}_{2}^{*},\,\mathcal{B},\, \mathcal{Y}\rangle$  be two
soft rough graphs of $G.$ Then AND operation of $\tilde{G_1}$ and
$\tilde{G_2}$ is denoted by $ \tilde{G_1} \curlywedge \tilde{G_2}$
and is defined as
 \[\tilde{G_{1}} \curlywedge \tilde{G_{2}}=\langle\mathfrak{F}_{*},
 \mathcal{K}_{*},\,\mathfrak{F}^{*},\,\mathcal{K}^{*},
  \,\mathcal{A}\times \mathcal{B}, \mathcal{X} \times
 \mathcal{Y}\rangle,\] where
$\mathfrak{F}_{*}(\mathcal{X} \times
\mathcal{Y})=\mathfrak{F}_{1*}(\mathcal{X})\cap
\mathfrak{F}_{2*}(\mathcal{Y}),$ $\mathcal{K}_{*}(\mathcal{X} \times
\mathcal{Y})=\mathcal{K}_{1*}(\mathcal{X})\cap
\mathcal{K}_{2*}(\mathcal{Y}),$ $\mathfrak{F}^{*}(\mathcal{X} \times
\mathcal{Y})=\mathfrak{F}_{1}^{*}(\mathcal{X})\cap
\mathfrak{F}_{2}^{*}(\mathcal{Y})$ and $\mathcal{K}^{*}(\mathcal{X}
\times \mathcal{Y})=\mathcal{K}_{1}^{*}(\mathcal{X})\cap
\mathcal{K}_{2}^{*}(\mathcal{Y}).$
\end{Definition}
\begin{Definition}
 Let
$\tilde{G_1}=\langle\mathfrak{F}_{1*},\,\mathfrak{F}_{1}^{*},\,\mathcal{K}_{1*},\,
\mathcal{K}_{1}^{*},\,\mathcal{A},\, \mathcal{X}\rangle$ and
$\tilde{G_2}=\langle\mathfrak{F}_{2*},\,\mathfrak{F}_{2}^{*},\,\mathcal{K}_{2*},\,
\mathcal{K}_{2}^{*},\,\mathcal{B},\, \mathcal{Y}\rangle$  be two
soft rough graphs of $G.$ Then OR operation of $\tilde{G_1}$ and
$\tilde{G_2}$ is denoted by $ \tilde{G_1} \curlyvee \tilde{G_2}$ and
is defined as
\[\tilde{G_{1}} \curlyvee \tilde{G_{2}}=\langle\mathfrak{F}_{*},
 \mathcal{K}_{*},\,\mathfrak{F}^{*},\,\mathcal{K}^{*},
  \,\mathcal{A}\times \mathcal{B},\, \mathcal{X} \times
 \mathcal{Y}\rangle,\] where
$\mathfrak{F}_{*}(\mathcal{X} \times
\mathcal{Y})=\mathfrak{F}_{1*}(\mathcal{X})\cup
\mathfrak{F}_{2*}(\mathcal{Y}),$ $\mathcal{K}_{*}(\mathcal{X} \times
\mathcal{Y})=\mathcal{K}_{1*}(\mathcal{X})\cup
\mathcal{K}_{2*}(\mathcal{Y}),$ $\mathfrak{F}^{*}(\mathcal{X} \times
\mathcal{Y})=\mathfrak{F}_{1}^{*}(\mathcal{X})\cup
\mathfrak{F}_{2}^{*}(\mathcal{Y})$ and $\mathcal{K}^{*}(\mathcal{X}
\times \mathcal{Y})=\mathcal{K}_{1}^{*}(\mathcal{X})\cup
\mathcal{K}_{2}^{*}(\mathcal{Y}).$
\end{Definition}
\begin{Example} Let $G=(V,E)$ be a graph with $V=\{v_{1}, v_{2}, v_{3}, v_{4}, v_{5}\}$
and $E=\{e_{1}, e_{2}, e_{3}, e_{4}, e_{5}, e_{6}, e_{7}, e_{8}\}$
as shown in \textup{FIGURE \ref{fig28}}. Let $\mathcal{A}=\{v_{2},
v_{5}\}$ be the set of parameters. We define an approximate function
$\mathfrak{F}_{1}: \mathcal{A} \rightarrow P(V)$ as
$$\mathfrak{F}_{1}(x)=\{y \in V:  xRy \Leftrightarrow y\in N(x)\}.$$
Then $\mathfrak{F} _{1}(v_{2})=\{v_{1}, v_{3}, v_{5}\},$
$\mathfrak{F}_{1}(v_{5})=\{v_{1}, v_{2}, v_{3}, v_{4}\}.$
\begin{figure}[!ht]
    \centerline
      {\includegraphics[width=4cm]{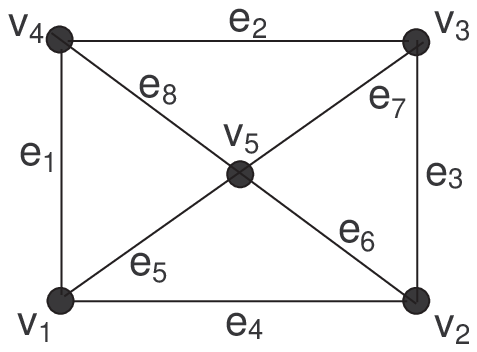}}
      \caption{}\label{fig28}
\end{figure}
Let $\mathcal{X}=\{v_{1}, v_{2}, v_{3}, v_{5}\}$ be the subset of
$V.$ The lower and upper soft rough approximations over $V$ are as
follows:
$$\mathfrak{F}_{1*}(\mathcal{X})=\{v_{1}, v_{3}, v_{5}\},\,\,\mathfrak{F}_{1}^{*}(\mathcal{X})=\{v_{1}, v_{2}, v_{3}, v_{4}, v_{5}\}$$
respectively. The lower and upper soft rough approximations over $E$
are as follows:
$$\mathcal{K}_{1*}(\mathcal{X})=\{e_{5}, e_{7}\},\,\,\mathcal{K}_{1}^{*}(\mathcal{X})=\{e_{1}, e_{2}, e_{3}, e_{4},
e_{5},e_6, e_{7},e_8\}$$ respectively. So,
$\tilde{G}_{1}=\{\mathcal{H}_{1*}(\mathcal{X}),
\mathcal{H}_{1}^{*}(\mathcal{X})\}$ be the soft rough graph of $G.$
Let $\mathcal{B}=\{v_{3}, v_{4}\}$ be the set of parameters and
$\mathcal{Y}=\{v_{1}, v_{2}, v_{4}, v_{5}\}$ be a subset of $V.$
 We define an approximate
function\, $\mathfrak{F}_{2}: \mathcal{B} \rightarrow P(V)$ as
$$\mathfrak{F}_{2}(x)=\{y \in V:  xRy \Leftrightarrow y\in N(x)\}.$$
Then $\mathfrak{F}_{2}(v_{3})=\{v_{2}, v_{4}, v_{5}\},$
$\mathfrak{F}_{2}(v_{4})=\{v_{1}, v_{3}, v_{5}\}.$ The lower and
upper soft rough approximations over $V$ are as follows:
$$\mathfrak{F}_{2*}(\mathcal{Y})=\{v_{2}, v_{4}, v_{5}\},\,\,\mathfrak{F}_{2}^{*}(\mathcal{Y})=\{v_{1}, v_{3}, v_{4}, v_{5}\}$$
respectively. The lower and upper soft rough approximations over $E$
are as follows:
$$\mathcal{K}_{2*}(\mathcal{Y})=\{e_{6}, e_{8}\},\,\,\mathcal{K}_{2}^{*}(\mathcal{Y})=\{e_{1}, e_{2}, e_{5}, e_{7},
e_{8}\}$$ respectively. So,
$\tilde{G}_{2}=\{\mathcal{H}_{2*}(\mathcal{Y}),
\mathcal{H}_{2}^{*}(\mathcal{Y})\}$ be the soft rough graph of $G.$
Now the AND operation of $\tilde{G}_{1}$ and
$\tilde{G}_{2}$ is given as follows:\\
$\mathfrak{F}_{*}(\mathcal{X} \times
\mathcal{Y})=\mathfrak{F}_{1*}(\mathcal{X}) \cap
\mathfrak{F}_{2*}(\mathcal{Y})=\{v_{5}\}$,
$\mathcal{K}_{*}(\mathcal{X} \times
\mathcal{Y})=\mathcal{K}_{1*}(\mathcal{X}) \cap
\mathcal{K}_{2*}(\mathcal{Y})=\emptyset$,
$\mathfrak{F}^{*}(\mathcal{X} \times
\mathcal{Y})=\mathfrak{F}_{1}^{*}(\mathcal{X}) \cap
\mathfrak{F}_{2}^{*}(\mathcal{Y})=\{v_{1}, v_{3}, v_{4}, v_{5}\}$,
$\mathcal{K}^{*}(\mathcal{X} \times
\mathcal{Y})=\mathcal{K}_{1}^{*}(\mathcal{X}) \cap
\mathcal{K}_{2}^{*}(\mathcal{Y})=\{e_{1}, e_{2}, e_{5},
e_{7},e_8\}.$\par Now the OR operation of $\tilde{G}_{1}$ and
$\tilde{G}_{2}$ is given as follows:\\ $\mathfrak{F}_{*}(\mathcal{X}
\times \mathcal{Y})=\mathfrak{F}_{1*}(\mathcal{X}) \cup
\mathfrak{F}_{2*}(\mathcal{Y})=\{v_{1}, v_{2}, v_{3}, v_{4},
v_{5}\}$, $\mathcal{K}_{*}(\mathcal{X} \times
\mathcal{Y})=\mathcal{K}_{1*}(\mathcal{X}) \cup
\mathcal{K}_{2*}(\mathcal{Y})=\{e_{5}, e_{6}, e_{7}, e_{8} \}$,
$\mathfrak{F}^{*}(\mathcal{X} \times
\mathcal{Y})=\mathfrak{F}_{1}^{*}(\mathcal{X}) \cup
\mathfrak{F}_{2}^{*}(\mathcal{Y})=\{v_{1}, v_{2}, v_{3}, v_{4},
v_{5}\}$, $\mathcal{K}^{*}(\mathcal{X} \times
\mathcal{Y})=\mathcal{K}_{1}^{*}(\mathcal{X}) \cup
\mathcal{K}_{2}^{*}(\mathcal{Y})=\{e_{1}, e_{2}, e_{3}, e_{4},
e_{5}, e_{6}, e_{7}, e_{8}\}.$
\end{Example}
\begin{Theorem}
Let
$\tilde{G_1}=\langle\mathfrak{F_{1*}},\,\mathfrak{F_{1}^{*}},\,\mathcal{K}
_{1*},\ \mathcal{K}_{1}^{*},\,\mathcal{A}, \mathcal{X}\rangle$ and
\,$\tilde{G_2}=\langle
\mathfrak{F_{2*}},\,\mathfrak{F_{2}^{*}},\,\mathcal{K}_{2*},\
\mathcal{K}_{2}^{*},\,\mathcal{B}, \mathcal{Y}\rangle$  be two soft
rough graphs of $G$ such that $\mathfrak{F}_{1*}(\mathcal{X}) \cap
\mathfrak{F}_{2*}(\mathcal{Y}) \neq \emptyset,$
$\mathfrak{F}_{1}^{*}(\mathcal{X}) \cap
\mathfrak{F}_{2}^{*}(\mathcal{Y}) \neq \emptyset,$
$\mathcal{K}_{1*}(\mathcal{X}) \cap \mathcal{K}_{2*}(\mathcal{Y})
\neq \emptyset$ and $\mathcal{K}_{1}^{*}(\mathcal{X}) \cap
\mathcal{K}_{2}^{*}(\mathcal{Y}) \neq \emptyset.$ Then their AND
operation is a soft rough graph of $G.$
\end{Theorem}
\begin{proof} Since $\tilde{G_{1}}$ is a soft rough graph of $G,$
$\mathcal{H}_{1*}(\mathcal{X})=(\mathfrak{F}_{1*}(\mathcal{X}),\,\mathcal{K}_{1*}(\mathcal{X}))$
and
$\mathcal{H}_{1}^*(\mathcal{X})=(\mathfrak{F}_{1}^{*}(\mathcal{X}),\,\mathcal{K}_{1}^{*}(\mathcal{X}))$
are subgraphs of $G.$ Since $\tilde{G_{2}}$ is a soft rough graph of
$G,$
$\mathcal{H}_{2*}(\mathcal{X})=(\mathfrak{F}_{2*}(\mathcal{X}),\,\mathcal{K}_{2*}(\mathcal{X}))$
and
$\mathcal{H}_{2}^*(\mathcal{X})=(\mathfrak{F}_{2}^{*}(\mathcal{X}),\,\mathcal{K}_{2}^{*}(\mathcal{X}))$
are subgraphs of $G.$ Let $(\mathfrak{F}_{*}(\mathcal{X} \times
\mathcal{Y}),\, \mathcal{K}_{*} (\mathcal{X} \times
\mathcal{Y}))=(\mathfrak{F}_{1*}(\mathcal{X})\,\cap\,\mathfrak{F}_{2*}(\mathcal{Y}),\mathcal{K}_{1*}(\mathcal{X})\,\cap\,\mathcal{K}_{2*}(\mathcal{Y}))$
and $(\mathfrak{F^{*}}(\mathcal{X} \times \mathcal{Y}
),\,\mathcal{K^{*}}(\mathcal{X} \times \mathcal{Y}))=
(\mathfrak{F}_{1}^{*}(\mathcal{X})\,\cap\,\mathfrak{F}_{2}^{*}(\mathcal{Y}),\,
\mathcal{K}_{1}^{*}(\mathcal{X})\,
\cap\,\mathcal{K}_{2}^{*}(\mathcal{Y})).$ Since
$(\mathfrak{F_{1*}}(\mathcal{X}),\,\mathcal{K}_{1*}(\mathcal{X})),$
$(\mathfrak{F}_{2*}(\mathcal{Y}),\mathcal{K}_{2*}(\mathcal{Y})),$
$(\mathfrak{F}_{1}^{*}(\mathcal{X}),\,\mathcal{K}_{1}^{*}(\mathcal{X})),$,
$(\mathfrak{F_{2}^{*}}(\mathcal{Y}),\,
\mathcal{K}_{2}^{*}(\mathcal{Y}))$ be the subgraphs of $G$ and by
assumption $\mathfrak{F}_{1*}(\mathcal{X}) \cap
\mathfrak{F}_{2*}(\mathcal{Y}) \neq \emptyset$,
$\mathfrak{F}_{1}^{*}(\mathcal{X}) \cap
\mathfrak{F}_{2}^{*}(\mathcal{Y})\neq \emptyset,$
$\mathcal{K}_{1*}(\mathcal{X}) \cap \mathcal{K}_{2*}(\mathcal{Y})
\neq \emptyset$ and $\mathcal{K}_{1}^{*}(\mathcal{X}) \cap
\mathcal{K}_{2}^{*}(\mathcal{Y}) \neq \emptyset.$ So,
$(\mathfrak{F_{*}}(\mathcal{X} \times
\mathcal{Y}),\,\mathcal{K}_{*}(\mathcal{X} \times \mathcal{Y}))$ and
$(\mathfrak{F}^{*}(\mathcal{X} \times
\mathcal{Y}),\,\mathcal{K}^{*}(\mathcal{X} \times \mathcal{Y}))$\,
are subgraphs of $G.$ Hence $\tilde{G_{1}} \curlywedge
\tilde{G_{2}}$ be a soft rough graph of $G.$\end{proof}
\begin{Theorem}Let
$\tilde{G_1}=\langle\mathfrak{F_{1*}},\,\mathfrak{F_{1}^{*}},\,\mathcal{K}
_{1*},\ \mathcal{K}_{1}^{*},\,\mathcal{A}, \mathcal{X}\rangle$ and
\,$\tilde{G_2}=\langle
\mathfrak{F_{2*}},\,\mathfrak{F_{2}^{*}},\,\mathcal{K}_{2*},\
\mathcal{K}_{2}^{*},\,\mathcal{B},\,\mathcal{Y}\rangle$  be two soft
rough graphs of \, $G$ such that $\mathfrak{F}_{1*}(\mathcal{X})
\cup \mathfrak{F}_{2*}(\mathcal{Y}) \neq \emptyset,$
$\mathfrak{F}_{1}^{*}(\mathcal{X}) \cup
\mathfrak{F}_{2}^{*}(\mathcal{Y}) \neq \emptyset,$
$\mathcal{K}_{1*}(\mathcal{X}) \cup \mathcal{K}_{2*}(\mathcal{Y})
\neq \emptyset$ and $\mathcal{K}_{1}^{*}(\mathcal{X}) \cup
\mathcal{K}_{2}^{*}(\mathcal{Y}) \neq \emptyset.$ Then their OR
operation is a soft rough graph of $G.$\end{Theorem}
\begin{proof}
The proof is obvious.\end{proof} The Cartesian product of two graphs
$G_1=(V_1, E_1)$ and $G_2=(V_2, E_2)$ is a graph $G=G_1\times
G_2=(V, E)$ where the vertex set of $G$ is the cartesian product of
$V_1$ and $V_2$ and edge set is defined as $E=\{(u, v),\,(u, w)\,|\,
u \in V_1, (v, w) \in E_2\} \cup \{(y ,x),\, (z, x) \,|\, x \in V_2,
(y, z) \in E_1 \}$.
 \begin{Definition}
 Let $\tilde{G_1}=\langle\mathfrak{F}_{1*},\,\mathfrak{F}_{1}^{*},\,\mathcal{K}_{1*},\,
\mathcal{K}_{1}^{*},\,\mathcal{A},\, \mathcal{X}\rangle$ and
$\tilde{G_2}=\langle\mathfrak{F}_{2*},\,\mathfrak{F}_{2}^{*},\,\mathcal{K}_{2*},\,
\mathcal{K}_{2}^{*},\,\mathcal{B},\, \mathcal{Y}\rangle$  be soft
rough graphs of $G_{1}$ and $G_{2}$ respectively. The Cartesian
product of $\tilde{G_1}$ and $\tilde{G_2}$ is denoted by $
\tilde{G_1} \ltimes \tilde{G_2}$ and is defined as
 $$ \tilde{G_{1}} \ltimes \tilde{G_{2}}=\langle\mathcal{L}_{*},
  \mathcal{L}^{*}, \, \mathcal{A}\times \mathcal{B},\,\mathcal{X} \times \mathcal{Y}\rangle,$$ where
$\mathcal{L}_{*}(\mathcal{X} \times
\mathcal{Y})=\mathcal{H}_{1*}(\mathcal{X}) \times
\mathcal{H}_{2*}(\mathcal{Y})$, $\mathcal{L}^{*}(\mathcal{X} \times
\mathcal{Y})=\mathcal{H}_{1}^{*}(\mathcal{X}) \times
\mathcal{H}_{2}^{*}(\mathcal{Y})$ and $\mathcal{H}_{1*}(\mathcal{X})
\times \mathcal{H}_{2*}(\mathcal{Y}),$
$\mathcal{H}_{1}^{*}(\mathcal{X}) \times
\mathcal{H}_{2}^{*}(\mathcal{Y})$ be the Cartesian product of
subgraphs $\mathcal{H}_{1*}(\mathcal{X})$,
$\mathcal{H}_{2*}(\mathcal{Y})$, $\mathcal{H}_{1}^*(\mathcal{X})$
and $\mathcal{H}_{2}^*(\mathcal{Y})$.
\end{Definition}
The lexicographic product of two graphs $G_1$ and $G_2$ is denoted
by $G_1 \circ G_2$, is the graph with vertex set $V(G_1)\times
V(G_2)$ = $\{(a,v)\ |\ a\in V(G_1)$ and $v\in V(G_2)\}$, where
$(a,v)$ is adjacent to $(b,w)$ whenever $ab\in E(G_1)$ or $a=b$ and
$vw\in E(G_2)$. For any vertex $a\in V(G_1)$ and $b\in V(G_2)$, we
define the vertex set $G_{2}(a)$ = $\{(a,v)\in V(G_1\circ G_2)\ |\
v\in V(G_2)\}$ and $G_{1}(b)$ = $\{(v,b)\in V(G_1\circ G_2)\ |\ v
\in V(G_1)\}$.
\begin{Definition}
 Let
$\tilde{G_1}=\langle\mathfrak{F}_{1*},\,\mathfrak{F}_{1}^{*},\,\mathcal{K}_{1*},\,
\mathcal{K}_{1}^{*},\,\mathcal{A},\, \mathcal{X}\rangle$ and
$\tilde{G_2}=\langle\mathfrak{F}_{2*},\,\mathfrak{F}_{2}^{*},\,\mathcal{K}_{2*},\,
\mathcal{K}_{2}^{*},\,\mathcal{B},\, \mathcal{Y}\rangle$  be soft
rough graphs of $G_{1}$ and $G_{2}$ respectively. The lexicographic
 product of $\tilde{G_1}$ and $\tilde{G_2}$ is denoted by $
\tilde{G_1} \odot \tilde{G_2}$ and is defined as
 $$ \tilde{G_{1}} \odot \tilde{G_{2}}=\langle\mathcal{N}_{*},
 \mathcal{N}^{*}, \, \mathcal{A}\times \mathcal{B},\,\mathcal{X} \times \mathcal{Y}\rangle,$$ where
$\mathcal{N}_{*}(\mathcal{X} \times
\mathcal{Y})=\mathcal{H}_{1*}(\mathcal{X}) \circ
\mathcal{H}_{2*}(\mathcal{Y})$, $\mathcal{N}^{*}(\mathcal{X} \times
\mathcal{Y})=\mathcal{H}_{1}^{*}(\mathcal{X}) \circ
\mathcal{H}_{2}^{*}(\mathcal{Y})$ and $\mathcal{H}_{1*}(\mathcal{X})
\circ \mathcal{H}_{2*}(\mathcal{Y}),$
$\mathcal{H}_{1}^{*}(\mathcal{X}) \circ
\mathcal{H}_{2}^{*}(\mathcal{Y})$ be the lexicographic product of
subgraphs $\mathcal{H}_{1*}(\mathcal{X})$,
$\mathcal{H}_{2*}(\mathcal{Y})$, $\mathcal{H}_{1}^*(\mathcal{X})$
and $\mathcal{H}_{2}^{*}(\mathcal{Y})$.
\end{Definition}
The join of two graphs $G_1=(V_1,\, E_1)$ and $G_2=(V_2,\, E_2)$ is
denoted by $G_1+G_2$ and is defined as the union of graphs obtained
by joining each vertex $v \in V_1$ to all vertices $w \in V_2$, i.e
$G_1 + G_2=(V_1 \cup V_2,\,E_1 \cup E_2\, \cup E)$, where $E$ is the
set of all edges joining the vertices of $V_1$ and $V_2.$
\begin{Definition}
 Let
$\tilde{G_1}=\langle\mathfrak{F}_{1*},\,\mathfrak{F}_{1}^{*},\,\mathcal{K}_{1*},\,
\mathcal{K}_{1}^{*},\,\mathcal{A},\, \mathcal{X}\rangle$ and
$\tilde{G_2}=\langle\mathfrak{F}_{2*},\,\mathfrak{F}_{2}^{*},\,\mathcal{K}_{2*},\,
\mathcal{K}_{2}^{*},\,\mathcal{B},\, \mathcal{Y}\rangle$ be soft
rough graphs of $G_{1}$ and $G_{2}$ respectively. The join of
$\tilde{G_1}$ and $\tilde{G_2}$ is denoted by $ \tilde{G_1} \oplus
\tilde{G_2}$ and is defined as
 $$ \tilde{G_{1}} \oplus \tilde{G_{2}}=\langle\mathcal{D}_{*},
 \mathcal{D}^{*}, \, \mathcal{A}\times\mathcal{B},\,\mathcal{X} \times
 \mathcal{Y}\rangle,$$ where
$\mathcal{D}_{*}(\mathcal{X} \times
\mathcal{Y})=\mathcal{H}_{1*}(\mathcal{X})+
\mathcal{H}_{2*}(\mathcal{Y})$, $\mathcal{D}^{*}(\mathcal{X} \times
\mathcal{Y})=\mathcal{H}_{1}^{*}(\mathcal{X})+
\mathcal{H}_{2}^{*}(\mathcal{Y})$ and $\mathcal{H}_{1*}(\mathcal{X})
+ \mathcal{H}_{2*}(\mathcal{Y}),$ $\mathcal{H}_{1}^{*}(\mathcal{X})
+ \mathcal{H}_{2}^{*}(\mathcal{Y})$ be the join of subgraphs
$\mathcal{H}_{1*}(\mathcal{X}),$ $\mathcal{H}_{2*}(\mathcal{Y}),$
$\mathcal{H}_{1}^{*}(\mathcal{X})$ and
$\mathcal{H}_{2}^{*}(\mathcal{Y})$.
\end{Definition}
The corona product of two graphs $G_1=(V_1,\, E_1)$ and $G_2=(V_2,\,
E_2)$ with $|G_1|=n_1$ and $|G_2|=n_2$ is denoted by $G_1\diamond
G_2$ and is defined as the graph obtained from $G_1$ and $G_2$ by
taking one copy of $G_1$ and $n_1$ copies of $G_2$ and joining by an
edge each vertex from the $i$th-copy of $G_2$ with the $i$th-vertex
of $G_1$.

\begin{Definition}
 Let
$\tilde{G_1}=\langle\mathfrak{F}_{1*},\,\mathfrak{F}_{1}^{*},\,\mathcal{K}_{1*},\,
\mathcal{K}_{1}^{*},\,\mathcal{A},\, \mathcal{X}\rangle$ and
$\tilde{G_2}=\langle\mathfrak{F}_{2*},\,\mathfrak{F}_{2}^{*},\,\mathcal{K}_{2*},\,
\mathcal{K}_{2}^{*},\,\mathcal{B},\, \mathcal{Y}\rangle$  be soft
rough graphs of $G_{1}$ and $G_{2}$ respectively. The corona product
of $\tilde{G_1}$ and $\tilde{G_2}$ is denoted by $ \tilde{G_1}
\circledcirc \tilde{G_2}$ and is defined as
 $$ \tilde{G_{1}} \circledcirc \tilde{G_{2}}=\langle\mathcal{M}_{*},
 \mathcal{M}^{*}, \, \mathcal{A}\times \mathcal{B},\,\mathcal{X} \times
 \mathcal{Y}\rangle,$$ where
$\mathcal{M}_{*}(\mathcal{X} \times
\mathcal{Y})=\mathcal{H}_{1*}(\mathcal{X}) \diamond
\mathcal{H}_{2*}(\mathcal{Y})$, $\mathcal{M}^{*}(\mathcal{X} \times
\mathcal{Y})=\mathcal{H}_{1}^{*}(\mathcal{X}) \diamond
\mathcal{H}_{2}^{*}(\mathcal{Y})$ and $\mathcal{H}_{1*}(\mathcal{X})
\diamond \mathcal{H}_{2*}(\mathcal{Y}),$
$\mathcal{H}_{1}^{*}(\mathcal{X}) \diamond
\mathcal{H}_{2}^{*}(\mathcal{Y})$ be the corona product of subgraphs
$\mathcal{H}_{1*}(\mathcal{X}),$ $\mathcal{H}_{2*}(\mathcal{Y}),$
$\mathcal{H}_{1}^{*}(\mathcal{X})$ and
$\mathcal{H}_{2}^{*}(\mathcal{Y})$.
\end{Definition}

\begin{Example} Let $G_{1}$ and $G_{2}$ be the graphs given
in \textup{FIGURE \ref{fig50}}. Let $V_{1}=\{a, b, c, d, e\}$ and
$V_{2}=\{f, g, h, k\}$ be the set of vertices of $G_{1}$ and $G_{2}$
respectively. Let $E_{1}=\{e_{1}, e_{2}, e_{3}, e_{4}\}$ and
$E_{2}=\{t_{1}, t_{2}, t_{3}\}$ be the set of edges of $G_{1}$ and
$G_{2}$ respectively. Let $\mathcal{A}=\{d,e\}\subseteq V_1$ and
$\mathcal{B}=\{g,k\}\subseteq V_2$ be the set of parameters. We
define approximate functions $\mathfrak{F}_{1}: \mathcal{A}
\rightarrow P(V_{1})$ and $\mathfrak{F}_{2}: \mathcal{B} \rightarrow
P(V_{2})$ by
$$\mathfrak{F}_{1}(x)=\{y \in V_{1}: xRy \Leftrightarrow y \in N(x)\},\,\,\mathfrak{F}_{2}(x)=\{y \in V_{2}: xRy \Leftrightarrow y \in
N[x]\}.$$
\begin{figure}[!ht]
    \centerline
      {\includegraphics[width= 6 cm]{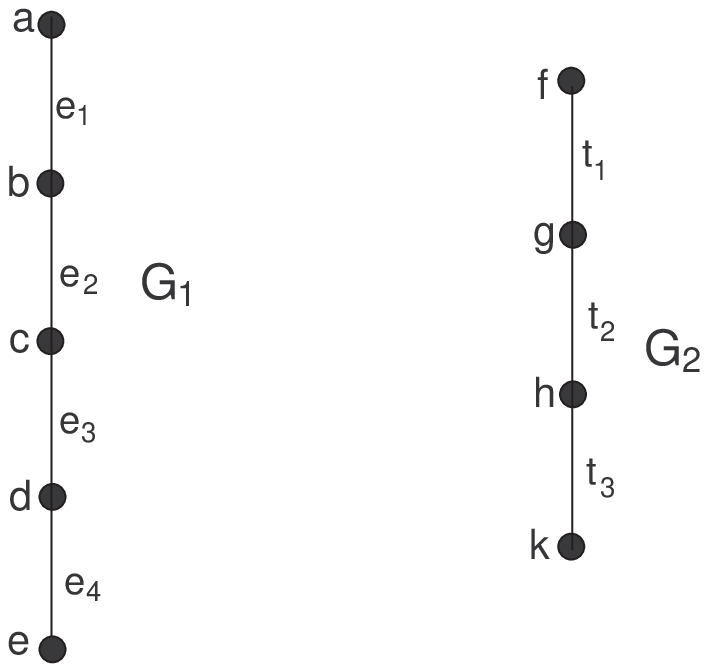}}
      \caption{}\label{fig50}
\end{figure}
Then $\mathfrak{F}_{1}(e)=\{d\},$ $\mathfrak{F}_{1}(d)=\{c,e\},$
$\mathfrak{F}_{2}(g)=\{f, g, h\},$ $ \mathfrak{F}_{2}(k)=\{h,k\}.$
Let $\mathcal{X}=\{b, c, d\}\subseteq V_1$ and
$\mathcal{Y}=\{h,k\}\subseteq V_2.$ The lower and upper vertex
approximations with respect to $\mathcal{X}$ and $\mathcal{Y}$ are
given as follows:
\begin{figure}[!ht]
    \centerline
      {\includegraphics[width=7 cm]{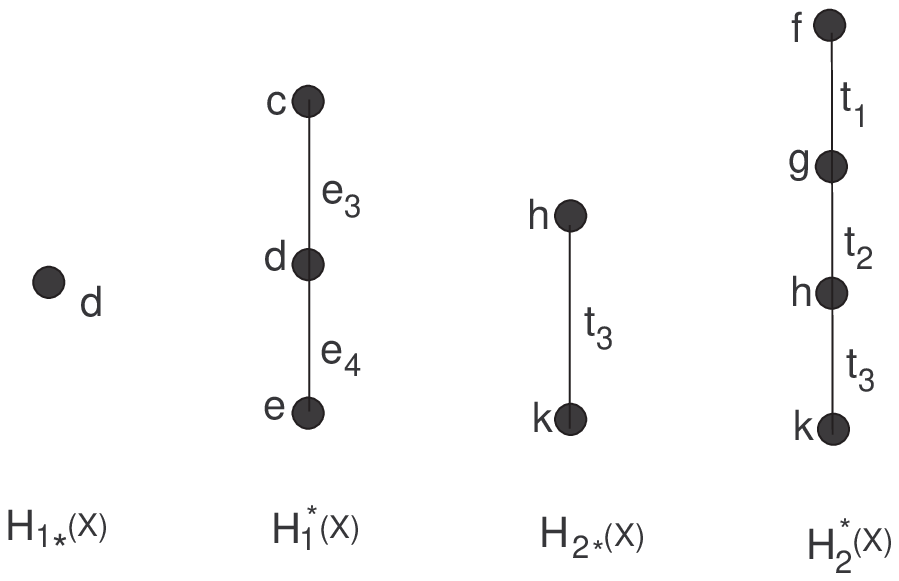}}
      \caption{}\label{fig51}
\end{figure}
$$\mathfrak{F}_{1*}(\mathcal{X})=\{d\},\,
\mathfrak{F}_{1}^{*}(\mathcal{X})=\{c, e, d\},\,
\mathfrak{F}_{2*}(\mathcal{Y})=\{k, h\},\,
\mathfrak{F}_{2}^{*}(\mathcal{Y})=\{f, g, h, k\}.$$ The lower and
upper edge approximations with respect to $\mathcal{X}$ and
$\mathcal{Y}$ are given as follows:
$$\mathcal{K}_{1}^{*}(\mathcal{X})=\{e_{3}, e_{4}\},\,
\mathcal{K}_{1*}(\mathcal{X})=\emptyset,\,
\mathcal{K}_{2}^{*}(\mathcal{Y})=\{t_{1}, t_{2},
t_{3}\},\,\mathcal{K}_{2*}(\mathcal{Y})=\{ t_{3}\}.$$ \textup{FIGURE
\ref{fig51}} shows
$\mathcal{H}_{1*}(\mathcal{X})=(\mathfrak{F}_{1*}(\mathcal{X}),\,\mathcal{K}_{1*}(\mathcal{X})),$
$\mathcal{H}_{2*}(\mathcal{Y})=(\mathfrak{F}_{2*}(\mathcal{Y}),\,\mathcal{K}_{2*}(\mathcal{Y})),$
$\mathcal{H}_{1}^{*}(\mathcal{X})=(\mathfrak{F}_{1}^{*}(\mathcal{X}),\,\mathcal{K}_{1}^{*}(\mathcal{X})),$
$\mathcal{H}_{2}^{*}(\mathcal{Y})=(\mathfrak{F}_{2}^{*}(\mathcal{Y}),\,\mathcal{K}_{2}^{*}(\mathcal{Y})).$
So, $\tilde{G}_{1}$ and $\tilde{G}_{2}$ are the soft rough graphs.
\textup{FIGURE \ref{figcartesian}} shows the subgraphs
$\mathcal{L_*}$ and $\mathcal{L^*}$.
\begin{figure}[!ht]
    \centerline
      {\includegraphics[width=8cm]{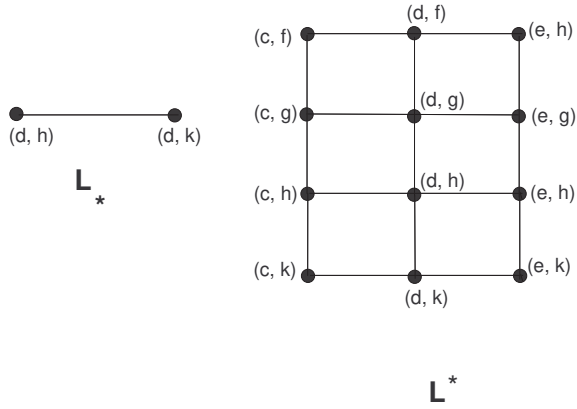}}
      \caption{Subgraphs $\mathcal{L_*}$ and $\mathcal{L^*}$}\label{figcartesian}
\end{figure}
 \textup{FIGURE \ref{fig52}} shows the subgraphs $\mathcal{N_*}$ and
 $\mathcal{N^*}$.
\begin{figure}[!ht]
    \centerline
      {\includegraphics[width=8cm]{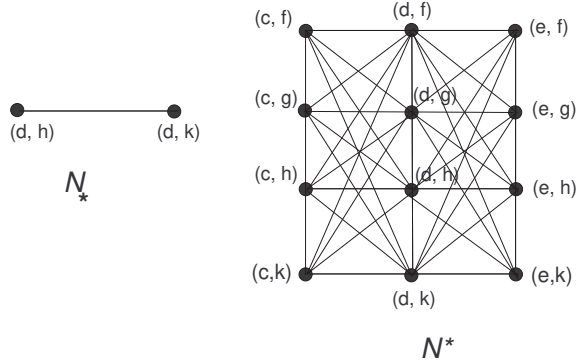}}
      \caption{Subgraphs $\mathcal{N_*}$ and $\mathcal{N^*}$}\label{fig52}
\end{figure}
\textup{FIGURE \ref{fig54}} shows the subgraphs $\mathcal{D_*}$ and
$\mathcal{D^*}$.
\begin{figure}[!ht]
    \centerline
      {\includegraphics[width=7cm]{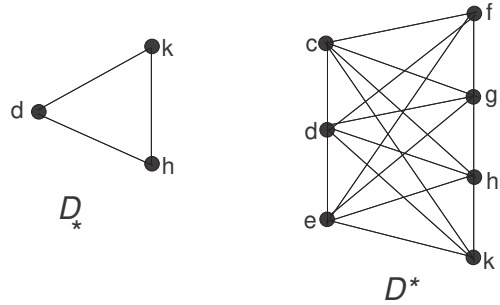}}
      \caption{Subgraphs $\mathcal{D_*}$ and $\mathcal{D^*}$}\label{fig54}
\end{figure}
\textup{FIGURE \ref{fig53}} shows the subgraphs $\mathcal{M_*}$ and
$\mathcal{M^*}$.
\begin{figure}[!ht]
    \centerline
      {\includegraphics[width=7cm]{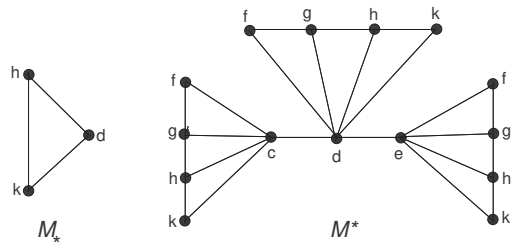}}
      \caption{Subgraphs $\mathcal{M_*}$ and $\mathcal{M^*}$}\label{fig53}
\end{figure}
\end{Example}
\begin{Theorem}
Let
$\tilde{G_1}=\langle\mathfrak{F}_{1*},\,\mathfrak{F}_{1}^{*},\,\mathcal{K}_{1*},\,
\mathcal{K}_{1}^{*},\,\mathcal{A},\, \mathcal{X}\rangle$ and
$\tilde{G_2}=\langle\mathfrak{F}_{2*},\,\mathfrak{F}_{2}^{*},\,\mathcal{K}_{2*},\,
\mathcal{K}_{2}^{*},\,\mathcal{B},\,\mathcal{Y}\rangle$  be soft
rough graphs of $G_{1}$ and $G_{2}$ respectively. Let $\mathcal{L}$
be the Cartesian product of $G_{1}$ and $G_{2}$ then $\tilde{G_{1}}
\ltimes
\tilde{G_{2}}=\langle\mathcal{L_{*}},\,\mathcal{L^{*}},\,\mathcal{A}\times
\mathcal{B},\,\mathcal{X} \times {Y}\rangle$ is a soft rough graph
of $\mathcal{L}.$
\begin{proof} The Cartesian product of two soft rough
graphs $\tilde{G_1}$  and $\tilde{G_2}$ is defined as $\tilde{G_{1}}
\ltimes
\tilde{G_{2}}=\langle\mathcal{L_{*}},\,\mathcal{L^{*}},\,\mathcal{A}\times
\mathcal{B},\,\mathcal{X} \times  \mathcal{Y}\rangle,$ where
$\mathcal{L_{*}}(\mathcal{X}\times
\mathcal{Y})=\mathcal{H}_{1*}(\mathcal{X}) \times
\mathcal{H}_{2*}(\mathcal{Y})$ and $\mathcal{L}^{*}(\mathcal{X
\times Y})=\mathcal{H}_{1}^{*}(\mathcal{X}) \times
\mathcal{H}_{2}^{*}(\mathcal{Y}).$ Since $\tilde{G_{1}}$ is a soft
rough graph of $G_{1},$ $\mathcal{H}_{1*}(\mathcal{X})$ and
$\mathcal{H}_{1}^{*}(\mathcal{X})$ are subgraphs of $G_{1},$ and
$\tilde{G_{2}}$ is a soft rough graph of $G_{2},$
$\mathcal{H}_{2*}(\mathcal{Y})$ and
$\mathcal{H}_{2}^{*}(\mathcal{Y})$ are subgraphs of $G_{2}$. So,
$\mathcal{L_{*}}(\mathcal{X \times Y})$ and
$\mathcal{L^{*}}(\mathcal{X \times Y})$ being the Cartesian product
of two subgraphs are subgraphs of $\mathcal{L}.$  Hence
$\tilde{G_{1}} \ltimes \tilde{G_{2}}$ is a soft rough graph of
$\mathcal{L}.$\end{proof}
\end{Theorem}
\begin{Theorem}
Let
$\tilde{G_1}=\langle\mathfrak{F}_{1*},\,\mathfrak{F}_{1}^{*},\,\mathcal{K}_{1*},\,
\mathcal{K}_{1}^{*},\,\mathcal{A},\, \mathcal{X}\rangle$ and
$\tilde{G_2}=\langle\mathfrak{F}_{2*},\,\mathfrak{F}_{2}^{*},\,\mathcal{K}_{2*},\,
\mathcal{K}_{2}^{*},\,\mathcal{B},\,\mathcal{Y}\rangle$  be soft
rough graphs of $G_{1}$ and $G_{2}$ respectively. Let $\mathcal{N}$
be the lexicographic product of $G_{1}$ and $G_{2}$ then
$\tilde{G_{1}} \odot
\tilde{G_{2}}=\langle\mathcal{N}_{*},\,\mathcal{N^{*}},\,\mathcal{A}\times\mathcal{B},\,\mathcal{X
\times Y}\rangle$ is a soft rough graph of $\mathcal{N}.$
\begin{proof}
The lexicographic product of two soft rough graphs $\tilde{G_1}$ and
$\tilde{G_2}$ is $\tilde{G_{1}} \odot
\tilde{G_{2}}=\langle\mathcal{N}_{*},\,\mathcal{N}^{*},\,\mathcal{A}\times
\mathcal{B},\,\mathcal{X} \times  \mathcal{Y}\rangle,$ where
$\mathcal{N}_{*}(\mathcal{X}\times
\mathcal{Y})=\mathcal{H}_{1*}(\mathcal{X}) \circ
\mathcal{H}_{2*}(\mathcal{Y})$ and $\mathcal{N}^{*}(\mathcal{X
\times Y})=\mathcal{H}_{1}^{*}(\mathcal{X})\, \circ\,
\mathcal{H}_{2}^{*}(\mathcal{Y}).$ Since $\tilde{G_{1}}$ is a soft
rough graph of $G_{1},$ $\mathcal{H}_{1*}(\mathcal{X})$ and
$\mathcal{H}_{1}^{*}(\mathcal{X})$ are subgraphs of $G_{1}.$  Since
$\tilde{G_{2}}$ is a soft rough graph of $G_{2},$
$\mathcal{H}_{2*}(\mathcal{Y})$ and
$\mathcal{H}_{2}^{*}(\mathcal{Y})$ are subgraphs of $G_{2}$. So,
$\mathcal{N}_{*}(\mathcal{X \times Y})$ and
$\mathcal{N}^{*}(\mathcal{X \times Y})$ being the lexicographic
product of two subgraphs are subgraphs of $\mathcal{N}.$  Hence
$\tilde{G_{1}} \odot \tilde{G_{2}}$ is a soft rough graph of
$\mathcal{N}.$\end{proof}
\end{Theorem}
\begin{Theorem}
Let
$\tilde{G_1}=\langle\mathfrak{F}_{1*},\,\mathfrak{F}_{1}^{*},\,\mathcal{K}_{1*},\,
\mathcal{K}_{1}^{*},\,\mathcal{A},\, \mathcal{X}\rangle$ and
$\tilde{G_2}=\langle\mathfrak{F}_{2*},\,\mathfrak{F}_{2}^{*},\,\mathcal{K}_{2*},\,
\mathcal{K}_{2}^{*},\,\mathcal{B},\,\mathcal{Y}\rangle$  be soft
rough graphs of $G_{1}$ and $G_{2}$ respectively. Let $\mathcal{D}$
be the join of $G_{1}$ and $G_{2}$ then $\tilde{G_{1}} \oplus
\tilde{G_{2}}=\langle\mathcal{D}_{*},\,\mathcal{D^{*}},\,\mathcal{A}\times
\mathcal{B},\,\mathcal{X \times Y}\rangle$ is a soft rough graph of
$\mathcal{N}.$\end{Theorem}
\begin{proof}
The join of two soft rough graphs $\tilde{G_1}$ and $\tilde{G_2}$ is
$\tilde{G_{1}}\oplus
\tilde{G_{2}}=\langle\mathcal{D}_{*},\,\mathcal{D}^{*},\,\mathcal{A}\times
\mathcal{B},\,\mathcal{X} \times  \mathcal{Y}\rangle,$ where
$\mathcal{D}_{*}(\mathcal{X}\times
\mathcal{Y})=\mathcal{H}_{1*}(\mathcal{X}) +
\mathcal{H}_{2*}(\mathcal{Y})$ and $\mathcal{D}^{*}(\mathcal{X
\times Y})=\mathcal{H}_{1}^{*}(\mathcal{X})+
\mathcal{H}_{2}^{*}(\mathcal{Y}).$ Since $\tilde{G_{1}}$ is a soft
rough graph of $G_{1},$ $\mathcal{H}_{1*}(\mathcal{X})$ and
$\mathcal{H}_{1}^{*}(\mathcal{X})$ are subgraphs of $G_{1}.$ Since
$\tilde{G_{2}}$ is a soft rough graph of $G_{2},$
$\mathcal{H}_{2*}(\mathcal{Y})$ and
$\mathcal{H}_{2}^{*}(\mathcal{Y})$ are subgraphs of $G_{2}$. So,
$\mathcal{D}_{*}(\mathcal{X \times Y})$ and
$\mathcal{D}^{*}(\mathcal{X \times Y})$ being the join  of two
subgraphs are subgraphs of $\mathcal{N}$. Hence $\tilde{G_{1}}
\oplus \tilde{G_{2}}$ is a soft rough graph of
$\mathcal{N}.$\end{proof}

\begin{Theorem}
Let
$\tilde{G_1}=\langle\mathfrak{F}_{1*},\,\mathfrak{F}_{1}^{*},\,\mathcal{K}_{1*},\,
\mathcal{K}_{1}^{*},\,\mathcal{A},\, \mathcal{X}\rangle$ and
$\tilde{G_2}=\langle\mathfrak{F}_{2*},\,\mathfrak{F}_{2}^{*},\,\mathcal{K}_{2*},\,
\mathcal{K}_{2}^{*},\,\mathcal{B},\,\mathcal{Y}\rangle$  be soft
rough graphs of $G_{1}$ and $G_{2}$ respectively. Let $\mathcal{M}$
be the corona product of $G_{1}$ and $G_{2}$ then $\tilde{G_{1}}
\circledcirc
\tilde{G_{2}}=\langle\mathcal{M}_{*},\,\mathcal{M^{*}},\,\mathcal{A}\times\mathcal{B},\,\mathcal{X
\times Y}\rangle$ is a soft rough graph of $\mathcal{M}.$
 \begin{proof}
The corona product of two soft rough graphs $\tilde{G_1}$ and
$\tilde{G_2}$ is $\tilde{G_{1}} \circledcirc
\tilde{G_{2}}=\langle\mathcal{M}_{*},\,\mathcal{M}^{*},\,\mathcal{A}\times
\mathcal{B},\,\mathcal{X} \times  \mathcal{Y}\rangle,$ where
$\mathcal{M}_{*}(\mathcal{X}\times
\mathcal{Y})=\mathcal{H}_{1*}(\mathcal{X}) \diamond
\mathcal{H}_{2*}(\mathcal{Y})$ and $\mathcal{M}^{*}(\mathcal{X
\times Y})=\mathcal{H}_{1}^{*}(\mathcal{X}) \diamond
\mathcal{H}_{2}^{*}(\mathcal{Y}).$ Since $\tilde{G_{1}}$ is a soft
rough graph of $G_{1},$ $\mathcal{H}_{1*}(\mathcal{X})$ and
$\mathcal{H}_{1}^{*}(\mathcal{X})$ are subgraphs of $G_{1}.$ Since
$\tilde{G_{2}}$ is a soft rough graph of $G_{2},$
$\mathcal{H}_{2*}(\mathcal{Y})$ and
$\mathcal{H}_{2}^{*}(\mathcal{Y})$ are subgraphs of $G_{2}$. So,
$\mathcal{M}_{*}(\mathcal{X \times Y})$ and
$\mathcal{M}^{*}(\mathcal{X \times Y})$ being the corona product of
two subgraphs are subgraphs of $\mathcal{N}.$ Hence $\tilde{G_{1}}
\circledcirc \tilde{G_{2}}$ is a soft rough graph of
$\mathcal{M}.$\end{proof}
\end{Theorem}


\begin{thebibliography}{99}
\bibitem{MAS2015} M. Akram, and S. Nawaz, Operations on soft graphs, Fuzzy Inf. Eng., 7(2015), 423-449.
\bibitem{FF2010}  F. Feng, C. Li, B. Davvaz and M. I. Ali,
Soft sets combined with fuzzy sets and rough sets: a tentative
approach, Soft Comput., 14(2010), 899-911.
\bibitem{FF2011} F. Feng, X. Liu, V. L. Fotea and Y. B. Jun, Soft sets and soft rough sets, Inform. Sci., 181(2011), 1125-1137.
\bibitem{GJ2013} J. Ghosh and T. K. Samanta, Rough Soft Sets and Rough Soft Groups, J. Hyperstructures, 2(1)(2013), 18-29.
\bibitem{HTK} H. Tong, L. Chang-jing and S. KAi-quan, Rough graph and its structure, J. Shandong Univ., 41(6)(2006), 46-50.
\bibitem{MRA2003}P. K. Maji, R. Biswas, and A. R. Roy, Soft set theory, Comput. Math. Appl., 45.(2003), 555-562.
\bibitem{MD1999}  D. Molodtsov, Soft set theory, first
results, Comput. Math. Appl., 37(1999), 19-31.
\bibitem{NM}E. K. R. Nagarajan and G. Meenambigai, An application of soft sets
to lattices, Kragujevac J. Math., 35(1)(2011), 75-87.
\bibitem{PZ1982}Z. Pawlak, Rough sets, Int. J. Comput. Inform. Sci., 11(1982), 41-356.
\bibitem{pawlk1}Z. Pawlak, Rough Sets: Theoretical Aspects of Reasoning about
Data, Kluwer Academic Publishers, Dordrecht, 1991.
\bibitem{pawlk2}Z. Pawlak,
A. Skowron, Rudiments of rough sets, Inform. Sci., 177(2007), 3-27.
\bibitem{pawlk3}Z. Pawlak, A. Skowron, Rough sets: some extensions, Inform. Sci., 177
(2007), 28-40.
\bibitem{pawlk4}Z. Pawlak, A. Skowron, Rough sets and
Boolean reasoning, Inform. Sci., 177(2007), 41-73.
\bibitem{TRB2014}R. K. Thumbakara, and B. George, Soft graphs, Gen. Math. Notes, 21(2)(2014), 75-86.





\end{thebibliography}
\end{document}